\pdfoutput=1
\pdftrailerid{}
\newcounter{corresponding}
\documentclass{extracta-temp-26}

\usepackage{latexsym}
\usepackage{amsmath}
\usepackage{amssymb}
\usepackage{amsfonts}
\usepackage{amsthm}
\usepackage[english]{babel}
\usepackage[cp850]{inputenc}
\usepackage{graphics}
\usepackage{mathtools}
\usepackage{enumitem}
\usepackage{xcolor}
\usepackage{tikz-cd}
\usepackage{microtype}
\usepackage{hyperref}
\usepackage[capitalise,nameinlink,noabbrev]{cleveref}

\numberwithin{equation}{section}

\newtheorem{teo}{Theorem}[section]
\newtheorem{lem}[teo]{Lemma}
\newtheorem{prop}[teo]{Proposition}

\theoremstyle{definition}
\newtheorem{defi}[teo]{Definition}
\newtheorem{exa}[teo]{Example}

\theoremstyle{remark}
\newtheorem{rem}[teo]{Remark}

\crefname{teo}{Theorem}{Theorems}
\Crefname{teo}{Theorem}{Theorems}
\crefname{lem}{Lemma}{Lemmas}
\Crefname{lem}{Lemma}{Lemmas}
\crefname{prop}{Proposition}{Propositions}
\Crefname{prop}{Proposition}{Propositions}
\crefname{defi}{Definition}{Definitions}
\Crefname{defi}{Definition}{Definitions}
\crefname{exa}{Example}{Examples}
\Crefname{exa}{Example}{Examples}
\crefname{rem}{Remark}{Remarks}
\Crefname{rem}{Remark}{Remarks}
\crefname{enumi}{}{}
\Crefname{enumi}{}{}
\creflabelformat{enumi}{#2\textup{(#1)}#3}

\setlist[itemize]{noitemsep,nolistsep,leftmargin=1.6em}
\setlist[enumerate]{noitemsep,nolistsep,leftmargin=1.9em,label=\arabic*.,ref=\arabic*}

\definecolor{typLink}{HTML}{1F4E79}
\hypersetup{
  colorlinks=true,
  linkcolor=typLink,
  citecolor=typLink,
  urlcolor=typLink,
  pdftitle={Groupoid Homology and Classifying-Space Homology Are Not Isomorphic},
  pdfauthor={Luciano Melodia},
  pdfcreator={},
  pdfproducer={}
}

\newcommand{\G}{\mathcal{G}}
\newcommand{\Z}{\mathbb{Z}}
\newcommand{\N}{\mathbb{N}}
\newcommand{\R}{\mathbb{R}}
\newcommand{\id}{\mathrm{id}}
\newcommand{\Hsing}{H^{\mathrm{sing}}}
\newcommand{\Csing}{C^{\mathrm{sing}}}
\newcommand{\dsing}{\partial^{\mathrm{sing}}}
\newcommand{\real}[1]{\lvert #1\rvert}
\newcommand{\gen}[1]{\langle #1\rangle}
\newcommand{\restr}[2]{#1\rvert_{#2}}
\DeclareMathOperator{\supp}{supp}
\DeclareMathOperator{\im}{im}
\DeclareMathOperator{\pr}{pr}
\DeclarePairedDelimiter{\abs}{\lvert}{\rvert}
\newcommand{\term}[1]{\textbf{#1}}
\newcommand{\step}[1]{\textbf{#1}}
\makeatletter
\newcommand{\keepwithnext}{\leavevmode\par\nobreak\@afterheading}
\makeatother
\newcommand{\subhead}[1]{\subsection*{#1}\keepwithnext}

\makeatletter
\def\ps@comienzo{%
  \let\@oddhead\@empty
  \let\@evenhead\@empty
  \let\@oddfoot\@empty
  \let\@evenfoot\@empty}
\makeatother

\title{Groupoid Homology and Classifying-Space Homology Are Not Isomorphic}

\author{Luciano Melodia\orcidid{0000-0002-7584-7287}}

\address{{\em Department of Mathematics \\[1ex] Friedrich-Alexander-Universit\"at Erlangen-N\"urnberg\\[1ex] Cauerstra{\ss}e 11, 91058 Erlangen, Germany \\[1ex]
    \correo{luciano.melodia@fau.de} }}

\mailcorresponding{luciano.melodia@fau.de}

\abstract{For an ample groupoid \(\G\), Matui-type groupoid homology \(H_\bullet(\G;\Z)\) is built from the nerve \(\G_\bullet\) via the Moore complex of compactly supported locally constant chains \(C_c(\G_n,\Z)\), with differential the alternating sum of pushforwards along the face maps. For a discrete group the theory agrees with the singular homology of the classifying space. For a totally disconnected locally compact Hausdorff space, viewed as a groupoid of units, it computes compactly supported cohomology instead, which is not a homotopy invariant. We make the resulting discrepancy explicit: for the unit groupoid on the Cantor set \(X\) we compute \(H_0(\G;\Z)\cong C(X,\Z)\), a countable group, whereas \(\Hsing_0(B\G;\Z)\cong\bigoplus_{x\in X}\Z\) has cardinality \(2^{\aleph_0}\). Cardinality alone separates the two groups, and it separates them in degree \(0\) already. In every positive degree they agree for this groupoid.}



\begin{document}

\maketitle

\section{Introduction and Statement}
A groupoid gives rise to two homotopy-theoretic invariants, both built from the same nerve \(\G_\bullet\), yet with opposite variance in the chain data.
On the one hand, Matui-type groupoid homology \(H_\bullet(\G;\Z)\) is the homology of the Moore complex (\cref{def:moore}) of compactly supported locally constant \(\Z\)-valued functions on the nerve, the differential being the alternating sum of fiberwise-sum pushforwards along the face maps.
The theory goes back to Crainic and Moerdijk \cite[Section~3]{crainic2000homology} and was recast for ample groupoids by Matui \cite[Definition~3.1]{matui2012homology} and recalled in \cite[Section~4]{farsi2019ample}.
On the other hand, the geometric realization of \(\G_\bullet\) is the classifying space \(B\G\), whose singular homology \(\Hsing_\bullet(B\G;\Z)\) records its homotopy type.

For a discrete group \(\Gamma\), regarded as a groupoid with a single unit, the two invariants agree: groupoid homology is then group homology \cite[3.5(1)]{crainic2000homology}, and group homology is the singular homology of the classifying space \(B\Gamma\) \cite[Chapter~II, p.~36]{brown1994cohomology}.
One might be tempted to expect the same for every ample groupoid.
Groups sit at one end of that class, and what happens at the other end argues against it.
A totally disconnected locally compact Hausdorff space \(X\), regarded as a groupoid consisting of units only, has groupoid homology \(H_n(X;\Z)\cong H_c^{-n}(X;\Z)\), the cohomology with compact supports placed in negative degrees \cite[3.5(3)]{crainic2000homology}.
In particular \(H_0(X;\Z)\cong C_c(X,\Z)\) and \(H_n(X;\Z)=0\) for \(n\ge 1\), as stated for zero-dimensional spaces in \cite[Proposition~4.4]{farsi2019ample}.
Compactly supported cohomology is not a homotopy invariant, so for a space there is no reason for the two theories to agree.
Any claim of agreement across the whole class has to account for both ends of it, groups and spaces.
We work out the second.

Degree \(0\) already separates the two theories.
Groupoid homology records there the locally constant integer-valued functions on the unit space, singular homology the path components.
On a totally disconnected unit space these two pieces of data are incomparable, and the following proposition turns the gap between them into an unconditional non-isomorphism, by cardinality.

\begin{prop}\label{prop:main}
Let \(X\coloneqq\{0,1\}^{\N}\) be the Cantor set and let \(\G\coloneqq(X\rightrightarrows X)\) be the unit groupoid.
Then
\begin{equation}\label{eq:main-moore}
H_0(\G;\Z)\cong C(X,\Z),\qquad
H_n(\G;\Z)=0\quad\text{for all } n\ge 1,
\end{equation}
whereas
\begin{equation}\label{eq:main-sing}
\Hsing_0(B\G;\Z)\cong\bigoplus_{x\in X}\Z.
\end{equation}
The first group is countable and the last has cardinality \(2^{\aleph_0}\).
In particular \(H_0(\G;\Z)\not\cong\Hsing_0(B\G;\Z)\).
\end{prop}

All notions entering \cref{prop:main} are defined in the body of the note: the unit groupoid in \cref{ex:unit}, groupoid homology \(H_\bullet(\G;\Z)\) in \cref{def:homology}, the classifying space \(B\G\) and singular homology \(\Hsing_\bullet\) in \cref{sec:classifying}, and the direct sum \(\bigoplus_{x\in X}\Z\) in \cref{def:dirsum}.
The groupoid homology of a space is not computed here for the first time. It appears as \cite[Proposition~4.4]{farsi2019ample} and is contained in \cite[3.5(3)]{crainic2000homology}.
We give the elementary proof because the comparison with the classifying space is the point of this note, and that comparison needs both sides in one explicit model.

\Cref{sec:pushforward} defines \'etale and ample groupoids, the nerve, the pushforward of compactly supported locally constant functions, and groupoid homology.
\Cref{sec:moore} computes \(H_\bullet(\G;\Z)\) for the Cantor unit groupoid.
\Cref{sec:classifying} identifies \(B\G\) with \(X\) and computes \(\Hsing_0\).
\Cref{sec:cardinality} separates the two groups by cardinality and completes the proof of \cref{prop:main}.

\section{Groupoid Homology via Pushforward and the Moore Complex}\label{sec:pushforward}
This section supplies the definitions used in \cref{prop:main}, in the order in which they build on one another: groupoids, \'etale and ample groupoids, the nerve, the fiberwise-sum pushforward, and the Moore complex, whose homology is groupoid homology.
Functoriality of the pushforward under local homeomorphisms (\cref{lem:pushforward}) is what makes the differential square to zero.

\subhead{Groupoids and the Nerve}

\begin{defi}\label{def:groupoid}
A \term{groupoid} \(\G\) is a small category all of whose morphisms are invertible.
We identify \(\G\) with its set of morphisms, called \term{arrows}, and its objects with the \term{units}, the identity morphisms \(\G^{(0)}\subseteq\G\).
We write \(r,s\colon\G\to\G^{(0)}\) for the maps assigning to an arrow the unit at its target and at its source, \(\G^{(2)}\coloneqq\{(\gamma,\eta)\in\G\times\G\mid s(\gamma)=r(\eta)\}\) for the set of \term{composable pairs}, \((\gamma,\eta)\mapsto\gamma\eta\) for composition and \(\gamma\mapsto\gamma^{-1}\) for inversion.
A \term{topological groupoid} is a groupoid \(\G\) equipped with a topology for which composition \(\G^{(2)}\to\G\) and inversion \(\G\to\G\) are continuous, where \(\G^{(2)}\) carries the subspace topology of \(\G\times\G\) and \(\G^{(0)}\) that of \(\G\).
\end{defi}

In a topological groupoid the maps \(r\) and \(s\) are continuous, since \(r(\gamma)=\gamma\gamma^{-1}\) and \(s(\gamma)=\gamma^{-1}\gamma\), and inversion is a homeomorphism, being a continuous involution.

\begin{defi}\label{def:etale}
An \term{\'etale groupoid} is a topological groupoid \(\G\) whose topology is locally compact and Hausdorff and whose range map \(r\colon\G\to\G^{(0)}\) is a local homeomorphism: every \(\gamma\in\G\) has an open neighborhood \(U\subseteq\G\) such that \(r(U)\) is open in \(\G^{(0)}\) and \(\restr{r}{U}\colon U\to r(U)\) is a homeomorphism.
A \term{bisection} of an \'etale groupoid is an open set \(U\subseteq\G\) such that \(\restr{r}{U}\) and \(\restr{s}{U}\) are homeomorphisms onto open subsets of \(\G^{(0)}\).
An \'etale groupoid is \term{ample} if \(\G^{(0)}\) has a basis of compact open sets.
\end{defi}

Since \(s=r\circ(\cdot)^{-1}\) and inversion is a homeomorphism, the source map of an \'etale groupoid is a local homeomorphism as well.
Every arrow of an \'etale groupoid lies in a bisection: intersect an open neighborhood on which \(r\) restricts to a homeomorphism onto an open set with one on which \(s\) does.
On the open intersection both restrictions remain injective, continuous and open onto their open images.
The convention for \'etale groupoids is that of \cite[Section~2]{matui2012homology}.
For ample groupoids compare \cite[Section~2.1]{farsi2019ample}, where a groupoid is called ample if it is \'etale with zero-dimensional unit space, equivalently if it has a basis of compact open bisections.

\begin{exa}\label{ex:unit}~
\begin{enumerate}
\item\label{it:ex-group} A discrete group \(\Gamma\) is an ample groupoid: all elements are arrows, the only unit is the neutral element, \(r\) and \(s\) are constant, and the topology is discrete.
\item\label{it:ex-unit} Every locally compact Hausdorff space \(X\) is a topological groupoid with \(\G=\G^{(0)}=X\), \(r=s=\id_X\) and \(xx=x\) for all \(x\in X\), the \term{unit groupoid} \(X\rightrightarrows X\).
It is \'etale, and it is ample exactly if \(X\) has a basis of compact open sets.
The Cantor set \(\{0,1\}^{\N}\) has such a basis, given by the sets of sequences with prescribed first \(n\) coordinates, so the unit groupoid of \cref{prop:main} is ample.
\item\label{it:ex-real} The unit groupoid of the real line \(\R\) is \'etale but not ample, since \(\R\) has no nonempty compact open subset.
\end{enumerate}
\end{exa}

For \(n\ge 0\) set \([n]\coloneqq\{0,1,\dots,n\}\).
A map \(\theta\colon[m]\to[n]\) is \term{monotone} if \(\theta(i)\le\theta(j)\) whenever \(i\le j\).
For \(n\ge 1\) and \(0\le i\le n\), the \(i\)-th \term{coface} \(\delta^n_i\colon[n-1]\to[n]\) is the unique monotone injection whose image omits \(i\).
Explicitly, \(\delta^n_i(k)=k\) for \(k<i\) and \(\delta^n_i(k)=k+1\) for \(k\ge i\).

\begin{lem}\label{lem:coface}
For \(n\ge 2\) and \(0\le i<j\le n\) one has
\begin{equation}\label{eq:coface-identity}
\delta^n_j\circ\delta^{n-1}_i=\delta^n_i\circ\delta^{n-1}_{j-1}.
\end{equation}
\end{lem}

\begin{proof}
Composites of monotone injections are monotone injections, and a monotone injection is determined by its image, so it suffices to compare images.
Since \(\delta^n_j(i)=i\) for \(i<j\), the left-hand side has image
\begin{equation}\label{eq:coface-left}
\begin{aligned}
\delta^n_j\bigl([n-1]\setminus\{i\}\bigr)
&=\bigl([n]\setminus\{j\}\bigr)\setminus\{i\}&&\quad\text{\(\delta^n_j\) injective}\\
&=[n]\setminus\{i,j\}.
\end{aligned}
\end{equation}
Since \(\delta^n_i(j-1)=j\) for \(j-1\ge i\), the right-hand side has image
\begin{equation}\label{eq:coface-right}
\begin{aligned}
\delta^n_i\bigl([n-1]\setminus\{j-1\}\bigr)
&=\bigl([n]\setminus\{i\}\bigr)\setminus\{j\}&&\quad\text{\(\delta^n_i\) injective}\\
&=[n]\setminus\{i,j\}
\end{aligned}
\end{equation}
as well.
\end{proof}

\begin{defi}\label{def:simplicial}
A \term{simplicial space} \(L_\bullet\) is a family \((L_n)_{n\ge 0}\) of topological spaces together with a continuous map \(\theta^*\colon L_n\to L_m\) for every monotone \(\theta\colon[m]\to[n]\), such that \((\id_{[n]})^*=\id_{L_n}\) for all \(n\ge 0\) and \((\theta\circ\theta')^*=\theta'^*\circ\theta^*\) for all monotone \(\theta'\colon[l]\to[m]\) and \(\theta\colon[m]\to[n]\).
A \term{morphism of simplicial spaces} \(f_\bullet\colon L_\bullet\to L'_\bullet\) is a family of continuous maps \(f_n\colon L_n\to L'_n\) with \(f_m\circ\theta^*=\theta^*\circ f_n\) for every monotone \(\theta\colon[m]\to[n]\).
It is an \term{isomorphism} if every \(f_n\) is a homeomorphism.
\end{defi}

If every \(f_n\) is a homeomorphism, the family \((f_n^{-1})_{n\ge 0}\) is again a morphism, since \(f_m\circ\theta^*=\theta^*\circ f_n\) gives \(\theta^*\circ f_n^{-1}=f_m^{-1}\circ\theta^*\).
An isomorphism of simplicial spaces has an inverse morphism.

\begin{defi}\label{def:nerve}
Let \(\G\) be a topological groupoid with unit space \(X\coloneqq\G^{(0)}\).
Set \(\G_0\coloneqq X\) and, for \(n\ge 1\),
\begin{equation}\label{eq:nerve-n}
\G_n\coloneqq\{(\gamma_1,\dots,\gamma_n)\in\G^n\mid s(\gamma_i)=r(\gamma_{i+1})\ \text{for }1\le i\le n-1\},
\end{equation}
with the subspace topology of \(\G^n\).
For \(g=(\gamma_1,\dots,\gamma_n)\in\G_n\) define units \(x_0(g)\coloneqq r(\gamma_1)\) and \(x_k(g)\coloneqq s(\gamma_k)\) for \(1\le k\le n\).
For \(g=x\in\G_0\) set \(x_0(g)\coloneqq x\).
For \(0\le a\le b\le n\) write \(\gamma_{a,b}\coloneqq\gamma_{a+1}\gamma_{a+2}\cdots\gamma_b\in\G\) for \(a<b\), and \(\gamma_{a,a}\coloneqq x_a(g)\).
The product is well defined by associativity, and induction on \(b-a\), respectively \(c-b\), from associativity and the unit laws gives
\begin{equation}\label{eq:partial-products}
r(\gamma_{a,b})=x_a(g),\qquad
s(\gamma_{a,b})=x_b(g),\qquad
\gamma_{a,b}\,\gamma_{b,c}=\gamma_{a,c}\quad\text{for } a\le b\le c.
\end{equation}
The \term{nerve} of \(\G\) is the family \(\G_\bullet=(\G_n)_{n\ge 0}\) with the structure maps
\begin{equation}\label{eq:nerve-structure}
\theta^*\colon\G_n\to\G_m,
\qquad
\theta^*(g)\coloneqq\bigl(\gamma_{\theta(0),\theta(1)},\gamma_{\theta(1),\theta(2)},\dots,\gamma_{\theta(m-1),\theta(m)}\bigr)
\end{equation}
for monotone \(\theta\colon[m]\to[n]\) with \(m\ge 1\), and \(\theta^*(g)\coloneqq x_{\theta(0)}(g)\) for \(m=0\).
The string \(\theta^*(g)\) is composable since \(s(\gamma_{\theta(j-1),\theta(j)})=x_{\theta(j)}(g)=r(\gamma_{\theta(j),\theta(j+1)})\).
For \(n\ge 1\) and \(0\le i\le n\) the \(i\)-th \term{face map} is \(d_i\coloneqq(\delta^n_i)^*\colon\G_n\to\G_{n-1}\).
\end{defi}

Unwinding the definition: for \(n\ge 2\),
\begin{equation}\label{eq:faces}
d_i(\gamma_1,\dots,\gamma_n)=
\begin{cases}
(\gamma_2,\dots,\gamma_n), & i=0,\\
(\gamma_1,\dots,\gamma_i\gamma_{i+1},\dots,\gamma_n), & 0<i<n,\\
(\gamma_1,\dots,\gamma_{n-1}), & i=n,
\end{cases}
\end{equation}
while for \(n=1\) the two face maps are \(d_0=s\) and \(d_1=r\colon\G_1=\G\to\G_0=\G^{(0)}\).
The map \(\theta^*\) for the monotone surjection \([n+1]\to[n]\) that takes the value \(i\) twice inserts the unit \(x_i(g)\) after the \(i\)-th entry.
These maps are the \term{degeneracy maps} \(s_i\colon\G_n\to\G_{n+1}\) for \(0\le i\le n\).

\begin{lem}\label{lem:nerve}
For every topological groupoid \(\G\) the nerve \(\G_\bullet\) is a simplicial space, and for all \(n\ge 2\) and \(0\le i<j\le n\),
\begin{equation}\label{eq:face-identity}
d_i\circ d_j=d_{j-1}\circ d_i\colon\G_n\to\G_{n-2}.
\end{equation}
\end{lem}

\begin{proof}\keepwithnext
\begin{itemize}
\item \step{Continuity of \(\theta^*\).}
For \(1\le k\le n\) the iterated multiplication \(\mu_k\colon\G_k\to\G\) with \(\mu_k(\gamma_1,\dots,\gamma_k)\coloneqq\gamma_1\cdots\gamma_k\) is continuous by induction on \(k\): the map \(\mu_1\) is the inclusion \(\G_1=\G\), and \(g\mapsto\bigl(\mu_{k-1}(\gamma_1,\dots,\gamma_{k-1}),\gamma_k\bigr)\) is a continuous map into \(\G^{(2)}\), because \(s(\gamma_1\cdots\gamma_{k-1})=s(\gamma_{k-1})=r(\gamma_k)\), so composing with the continuous multiplication \(\G^{(2)}\to\G\) gives \(\mu_k\).
Every component of \(\theta^*\) is a \(\mu_k\) applied to a subtuple of consecutive entries, or a unit \(x_a(g)\in\{r(\gamma_{a+1}),s(\gamma_a)\}\), or, for \(n=0\), the point \(g\) itself, hence continuous.
A map into the subspace \(\G_m\subseteq\G^m\) is continuous as soon as its components are.
\item \step{Functoriality.}
One has \((\id_{[n]})^*=\id\) since \(\gamma_{k-1,k}=\gamma_k\).
Let \(\theta'\colon[l]\to[m]\) and \(\theta\colon[m]\to[n]\) be monotone, \(g\in\G_n\), and \(h\coloneqq\theta^*(g)\).
The units of \(h\) are \(x_k(h)=x_{\theta(k)}(g)\): for \(m=0\) this reads \(x_0(h)=h=x_{\theta(0)}(g)\) by definition, and for \(m\ge 1\) one has \(x_0(h)=r(\gamma_{\theta(0),\theta(1)})=x_{\theta(0)}(g)\) and \(x_k(h)=s(\gamma_{\theta(k-1),\theta(k)})=x_{\theta(k)}(g)\) for \(1\le k\le m\).
Writing \(\eta_{a,b}\) for the partial products formed from \(h\), the identity \(\gamma_{a,b}\gamma_{b,c}=\gamma_{a,c}\) telescopes to
\begin{equation}\label{eq:eta-telescope}
\begin{aligned}
\eta_{a,b}
&=\gamma_{\theta(a),\theta(a+1)}\cdots\gamma_{\theta(b-1),\theta(b)}&&\quad\text{by \eqref{eq:nerve-structure}}\\
&=\gamma_{\theta(a),\theta(b)},&&\quad\text{by \eqref{eq:partial-products}}
\end{aligned}
\end{equation}
including the case \(a=b\), where \(\eta_{a,a}=x_a(h)=x_{\theta(a)}(g)=\gamma_{\theta(a),\theta(a)}\).
Hence
\begin{equation}\label{eq:nerve-functorial}
\begin{aligned}
\theta'^*(\theta^*(g))
&=\bigl(\eta_{\theta'(j-1),\theta'(j)}\bigr)_{1\le j\le l}&&\quad\text{by \eqref{eq:nerve-structure}}\\
&=\bigl(\gamma_{\theta\theta'(j-1),\theta\theta'(j)}\bigr)_{1\le j\le l}&&\quad\text{by \eqref{eq:eta-telescope}}\\
&=(\theta\circ\theta')^*(g),&&\quad\text{by \eqref{eq:nerve-structure}}
\end{aligned}
\end{equation}
and for \(l=0\) likewise \(\theta'^*(h)=x_{\theta'(0)}(h)=x_{\theta\theta'(0)}(g)=(\theta\circ\theta')^*(g)\).
\item \step{Face identity.}
By functoriality and \cref{lem:coface},
\begin{equation}\label{eq:face-identity-proof}
\begin{aligned}
d_i\circ d_j
&=(\delta^{n}_j\circ\delta^{n-1}_i)^*&&\quad\text{by \eqref{eq:nerve-functorial}}\\
&=(\delta^{n}_i\circ\delta^{n-1}_{j-1})^*&&\quad\text{by \eqref{eq:coface-identity}}\\
&=d_{j-1}\circ d_i,&&\quad\text{by \eqref{eq:nerve-functorial}}
\end{aligned}
\end{equation}
which is \eqref{eq:face-identity}.\qedhere
\end{itemize}
\end{proof}

\begin{lem}\label{lem:nervefacts}
Let \(\G\) be an ample groupoid. Then:
\begin{enumerate}
\item\label{it:nerve-lch} for every \(n\ge 0\) the space \(\G_n\) is locally compact, Hausdorff and totally disconnected,
\item\label{it:nerve-face} for every \(n\ge 1\) and \(0\le i\le n\) the face map \(d_i\colon\G_n\to\G_{n-1}\) is a local homeomorphism.
\end{enumerate}
\end{lem}

\begin{proof}\keepwithnext
\begin{itemize}
\item \step{\(\G\) has a basis of compact open sets.}
Let \(\gamma\in\G\) and let \(\Omega\subseteq\G\) be an open neighborhood of \(\gamma\).
Choose an open \(U\subseteq\Omega\) with \(\gamma\in U\) such that \(\restr{r}{U}\) is a homeomorphism onto the open set \(r(U)\subseteq\G^{(0)}\) (\cref{def:etale}).
Since \(\G\) is ample, there is a compact open \(V\subseteq\G^{(0)}\) with \(r(\gamma)\in V\subseteq r(U)\).
Then \((\restr{r}{U})^{-1}(V)\) is a compact open neighborhood of \(\gamma\) contained in \(\Omega\).
\item \step{\(\G\) is totally disconnected.}
Let \(\gamma\neq\eta\) in \(\G\).
As \(\G\) is Hausdorff, there is an open \(\Omega\ni\gamma\) with \(\eta\notin\Omega\), and inside it a compact open \(A\) with \(\gamma\in A\subseteq\Omega\).
Compact subsets of Hausdorff spaces are closed, so \(A\) is clopen.
Every \(S\subseteq\G\) containing \(\gamma\) and \(\eta\) is disconnected by \(S\cap A\) and \(S\setminus A\), so the connected components of \(\G\) are singletons.
\item \step{The spaces \(\G_n\).}
The product \(\G^n\) is locally compact Hausdorff, being a finite product of such spaces, and totally disconnected, since the components of a product are the products of components.
The subset \(\G_n\subseteq\G^n\) is closed: it is the intersection of the sets \(\{g\in\G^n\mid s(\gamma_i)=r(\gamma_{i+1})\}\) for \(1\le i\le n-1\), each the preimage of the closed diagonal of \(\G^{(0)}\times\G^{(0)}\) under the continuous map \(g\mapsto(s(\gamma_i),r(\gamma_{i+1}))\).
A closed subspace of a locally compact Hausdorff space is locally compact Hausdorff, and every subspace of a totally disconnected space is totally disconnected.
The space \(\G_0=\G^{(0)}\) is a subspace of \(\G\), hence Hausdorff and totally disconnected, and it is locally compact because it has a basis of compact open sets.
\item \step{Outer faces.}
For \(n=1\), \(d_0=s\) and \(d_1=r\) are local homeomorphisms by \cref{def:etale} and the paragraph following it.
Let \(n\ge 2\) and \(g=(\gamma_1,\dots,\gamma_n)\in\G_n\), and choose a bisection \(U\ni\gamma_1\).
Then \(\mathcal U\coloneqq(U\times\G^{n-1})\cap\G_n\) is an open neighborhood of \(g\) and \(\mathcal V\coloneqq\{(\eta_1,\dots,\eta_{n-1})\in\G_{n-1}\mid r(\eta_1)\in s(U)\}\) is open in \(\G_{n-1}\).
The map \(\psi\colon\mathcal V\to\mathcal U\) given by \(\psi(\eta_1,\dots,\eta_{n-1})\coloneqq\bigl((\restr{s}{U})^{-1}(r(\eta_1)),\eta_1,\dots,\eta_{n-1}\bigr)\) is well defined and continuous, and \(d_0(\mathcal U)\subseteq\mathcal V\) because \(r(\gamma_2)=s(\gamma_1)\in s(U)\).
Moreover \(d_0\circ\psi=\id_{\mathcal V}\), and \(\psi\circ\restr{d_0}{\mathcal U}=\id_{\mathcal U}\) since \(\restr{s}{U}\) is injective and \(\gamma_1\in U\).
Hence \(\restr{d_0}{\mathcal U}\colon\mathcal U\to\mathcal V\) is a homeomorphism onto an open set, so \(d_0\) is a local homeomorphism.
The face \(d_n\) is treated symmetrically, with a bisection \(U'\ni\gamma_n\) and the entry \((\restr{r}{U'})^{-1}(s(\eta_{n-1}))\) appended.
\item \step{Inner faces.}
Let \(n\ge 2\) and \(0<i<n\), and choose bisections \(U\ni\gamma_i\) and \(V\ni\gamma_{i+1}\).
On the open set \(r^{-1}(r(U))\subseteq\G\) define the continuous maps \(\lambda(\zeta)\coloneqq(\restr{r}{U})^{-1}(r(\zeta))\in U\) and \(\Phi(\zeta)\coloneqq\lambda(\zeta)^{-1}\zeta\), the latter defined since \(s(\lambda(\zeta)^{-1})=r(\lambda(\zeta))=r(\zeta)\).
Set \(UV\coloneqq\Phi^{-1}(V)\), an open subset of \(\G\).
Every \(\zeta\in UV\) factors as \(\zeta=\lambda(\zeta)\,\Phi(\zeta)\) with \(\lambda(\zeta)\in U\) and \(\Phi(\zeta)\in V\), since
\begin{equation}\label{eq:inner-factor}
\begin{aligned}
\lambda(\zeta)\bigl(\lambda(\zeta)^{-1}\zeta\bigr)
&=\bigl(\lambda(\zeta)\lambda(\zeta)^{-1}\bigr)\zeta&&\quad\text{associativity}\\
&=r(\zeta)\,\zeta&&\quad\text{\(r(\lambda(\zeta))=r(\zeta)\)}\\
&=\zeta,&&\quad\text{unit law}
\end{aligned}
\end{equation}
and the factorization into an element of \(U\) times an element of \(V\) is unique: \(\zeta=\gamma\eta\) with \(\gamma\in U\) and \(\eta\in V\) forces \(r(\gamma)=r(\zeta)\), hence \(\gamma=\lambda(\zeta)\) by injectivity of \(\restr{r}{U}\), and then \(\eta=\gamma^{-1}\zeta=\Phi(\zeta)\).
Now \(\mathcal U\coloneqq(\G^{i-1}\times U\times V\times\G^{n-i-1})\cap\G_n\) is an open neighborhood of \(g\), and \(\mathcal V\coloneqq\{(\eta_1,\dots,\eta_{n-1})\in\G_{n-1}\mid\eta_i\in UV\}\) is open in \(\G_{n-1}\).
For \(h=(\eta_1,\dots,\eta_{n-1})\in\mathcal V\) the string
\begin{equation}\label{eq:inner-psi}
\psi(h)\coloneqq(\eta_1,\dots,\eta_{i-1},\lambda(\eta_i),\Phi(\eta_i),\eta_{i+1},\dots,\eta_{n-1})
\end{equation}
is composable, since \(s(\eta_{i-1})=r(\eta_i)=r(\lambda(\eta_i))\), \(s(\lambda(\eta_i))=r(\lambda(\eta_i)^{-1})=r(\Phi(\eta_i))\) and \(s(\Phi(\eta_i))=s(\eta_i)=r(\eta_{i+1})\), the first condition being vacuous for \(i=1\) and the last for \(i=n-1\), and \(\psi\) is continuous.
Then \(d_i(\mathcal U)\subseteq\mathcal V\), because \(\lambda(\gamma_i\gamma_{i+1})=\gamma_i\) and \(\Phi(\gamma_i\gamma_{i+1})=\gamma_i^{-1}\gamma_i\gamma_{i+1}=\gamma_{i+1}\in V\).
The same computation gives \(\psi\circ\restr{d_i}{\mathcal U}=\id_{\mathcal U}\), and \(d_i\circ\psi=\id_{\mathcal V}\) by the factorization \(\lambda(\eta_i)\,\Phi(\eta_i)=\eta_i\).
Hence \(\restr{d_i}{\mathcal U}\colon\mathcal U\to\mathcal V\) is a homeomorphism onto an open set.\qedhere
\end{itemize}
\end{proof}

\subhead{Pushforward}
The chains of the Moore complex are compactly supported locally constant functions, and its differential moves them along the face maps.
We now define that operation and establish its basic properties.

\begin{defi}\label{def:Cc}
Let \(Y\) be locally compact Hausdorff and totally disconnected.
Set
\begin{equation}\label{eq:Cc}
C_c(Y,\Z)\coloneqq\{f\colon Y\to\Z\mid f\ \text{locally constant},\ \supp(f)\ \text{compact}\}.
\end{equation}
Since \(f\) is locally constant and \(\Z\) is discrete, \(\{f\neq0\}=f^{-1}(\Z\setminus\{0\})\) is clopen, so \(\supp(f)=\{f\neq0\}\) is a \emph{compact open} subset of \(Y\).

If \(Z\) is locally compact Hausdorff and totally disconnected, \(p\colon Y\to Z\) is a local homeomorphism and \(f\in C_c(Y,\Z)\), define \(p_*f\colon Z\to\Z\) by
\begin{equation}\label{eq:pushforward}
(p_*f)(z)\coloneqq\sum_{y\in p^{-1}(z)}f(y).
\end{equation}
The sum is finite: \(p^{-1}(z)\) is closed (as \(Z\) is Hausdorff and \(p\) continuous) and discrete (as \(p\) is a local homeomorphism), so \(\supp(f)\cap p^{-1}(z)\) is a compact discrete set, hence finite, and only its points contribute.
\end{defi}

\begin{lem}\label{lem:pushforward}
Let \(Y\), \(Z\) and \(W\) be locally compact Hausdorff and totally disconnected, and let \(p\colon Y\to Z\) be a local homeomorphism. Then:
\begin{enumerate}
\item\label{it:push-cc} \(p_*f\in C_c(Z,\Z)\) for every \(f\in C_c(Y,\Z)\),
\item\label{it:push-funct} if \(q\colon Z\to W\) is a local homeomorphism, then \((q\circ p)_*=q_*\,p_*\),
\item\label{it:push-homeo} if \(p\) is a homeomorphism, then \(p_*=(p^{-1})^*\), the pullback \(f\mapsto f\circ p^{-1}\).
\end{enumerate}
\end{lem}

\begin{proof}
Fix \(f\in C_c(Y,\Z)\) and put \(K\coloneqq\supp(f)\), a compact open set with \(\restr{f}{Y\setminus K}=0\).
\begin{itemize}
\item \step{Local constancy of \(p_*f\).}
Fix \(z\in Z\) and enumerate the (finite) fiber of the support,
\begin{equation}\label{eq:fiber-enum}
K\cap p^{-1}(z)=\{y_1,\dots,y_m\},\qquad m\ge 0.
\end{equation}
Because \(p\) is a local homeomorphism, \(f\) is locally constant, and \(Y\) is Hausdorff, we may choose \emph{pairwise disjoint} open sets \(V_1,\dots,V_m\) with \(y_j\in V_j\) such that for each \(j\) the restriction \(\restr{p}{V_j}\colon V_j\to p(V_j)\) is a homeomorphism onto an open set and \(\restr{f}{V_j}\equiv f(y_j)\).
Indeed, a local-homeomorphism chart at \(y_j\) gives the first condition on a neighborhood, intersecting with the clopen set on which \(f\) is constantly \(f(y_j)\) gives the second, and Hausdorffness lets us shrink the finitely many \(V_j\) to be disjoint.
Set \(K'\coloneqq K\setminus(V_1\cup\dots\cup V_m)\).
This is closed in \(K\), hence compact, and contains no point of \(p^{-1}(z)\), since every point of \(K\) over \(z\) is some \(y_j\in V_j\).
Thus \(z\notin p(K')\), and \(p(K')\) is compact, hence closed in \(Z\).
Consequently
\begin{equation}\label{eq:Wprime}
W'\coloneqq\bigl(Z\setminus p(K')\bigr)\cap\bigcap_{j=1}^m p(V_j)
\end{equation}
is an open neighborhood of \(z\) (each \(p(V_j)\ni z\) is open, and for \(m=0\) the empty intersection is read as \(Z\)).
We claim \(p_*f\equiv(p_*f)(z)\) on \(W'\).
Fix \(z'\in W'\).
For each \(j\), since \(z'\in p(V_j)\) and \(\restr{p}{V_j}\) is bijective onto \(p(V_j)\), there is a unique \(y_j'\in V_j\) with \(p(y_j')=z'\), and \(f(y_j')=f(y_j)\).
The \(y_j'\) are distinct as the \(V_j\) are disjoint.
Every \(y\in p^{-1}(z')\) with \(f(y)\neq0\) lies in \(K\), and since \(z'\notin p(K')\) it cannot lie in \(K'\), so it lies in some \(V_j\) and hence equals \(y_j'\).
Therefore
\begin{equation}\label{eq:pf-locally-constant}
\begin{aligned}
(p_*f)(z')
&=\sum_{y\in p^{-1}(z')}f(y)&&\quad\text{by \eqref{eq:pushforward}}\\
&=\sum_{j=1}^m f(y_j')&&\quad\text{the nonzero terms sit at the \(y_j'\)}\\
&=\sum_{j=1}^m f(y_j)&&\quad\text{\(f(y_j')=f(y_j)\)}\\
&=(p_*f)(z).&&\quad\text{by \eqref{eq:pushforward}}
\end{aligned}
\end{equation}
As \(z\in Z\) was arbitrary, \(p_*f\) is locally constant.
\item \step{Compact support of \(p_*f\).}
If \(z\notin p(K)\), then \(p^{-1}(z)\cap K=\varnothing\) and \((p_*f)(z)=0\), so \(\{p_*f\neq0\}\subseteq p(K)\).
Since \(p(K)\) is compact, hence closed, we get \(\supp(p_*f)\subseteq p(K)\), so \(\supp(p_*f)\) is compact.
With local constancy this gives \(p_*f\in C_c(Z,\Z)\).
\item \step{Functoriality.}
Let \(q\colon Z\to W\) be a local homeomorphism.
Then \(q\circ p\) is a local homeomorphism (shrink a chart of \(p\) at each point until its image lies inside a chart of \(q\)), so \((q\circ p)_*\) is defined.
Fix \(w\in W\).
The fibers \(\{p^{-1}(z)\}_{z\in q^{-1}(w)}\) partition \((q\circ p)^{-1}(w)=p^{-1}\bigl(q^{-1}(w)\bigr)\), and \(K\cap(q\circ p)^{-1}(w)\) is finite by the argument of \cref{def:Cc} applied to \(q\circ p\).
Hence the rearrangement
\begin{equation}\label{eq:pf-functorial}
\begin{aligned}
\bigl((q\circ p)_*f\bigr)(w)
&=\sum_{y\in(q\circ p)^{-1}(w)}f(y)&&\quad\text{by \eqref{eq:pushforward}}\\
&=\sum_{z\in q^{-1}(w)}\ \sum_{y\in p^{-1}(z)}f(y)&&\quad\text{partition of \((q\circ p)^{-1}(w)\)}\\
&=\sum_{z\in q^{-1}(w)}(p_*f)(z)&&\quad\text{by \eqref{eq:pushforward}}\\
&=\bigl(q_*(p_*f)\bigr)(w)&&\quad\text{by \eqref{eq:pushforward}}
\end{aligned}
\end{equation}
is a finite reindexing, valid termwise.
Thus \((q\circ p)_*=q_*\,p_*\).
\item \step{Homeomorphisms.}
If \(p\) is a homeomorphism, \(p^{-1}(z)\) is the singleton \(\{p^{-1}(z)\}\), so \((p_*f)(z)=f\bigl(p^{-1}(z)\bigr)=\bigl((p^{-1})^*f\bigr)(z)\). Thus \(p_*=(p^{-1})^*\).\qedhere
\end{itemize}
\end{proof}

\subhead{The Moore Complex}

\begin{defi}\label{def:moore}
Let \(\G\) be an ample groupoid.
By \cref{lem:nervefacts}\labelcref{it:nerve-lch} each \(\G_n\) is locally compact Hausdorff and totally disconnected, so \cref{def:Cc} applies to it, and by \cref{lem:nervefacts}\labelcref{it:nerve-face} and \cref{lem:pushforward}\labelcref{it:push-cc} each face map induces a pushforward \((d_i)_*\colon C_c(\G_n,\Z)\to C_c(\G_{n-1},\Z)\).
The \term{Moore complex} of \(\G\) is
\begin{equation}\label{eq:moore-complex}
\begin{aligned}
C_n(\G;\Z)&\coloneqq C_c(\G_n,\Z),\\
\partial_n&\coloneqq\sum_{i=0}^n(-1)^i(d_i)_*\colon C_n(\G;\Z)\to C_{n-1}(\G;\Z)\quad(n\ge 1),
\qquad
\partial_0\coloneqq 0,
\end{aligned}
\end{equation}
where \(C_{-1}(\G;\Z)\coloneqq 0\).
\end{defi}

In words, the Moore complex is the unnormalized chain complex of the simplicial abelian group \(n\mapsto C_c(\G_n,\Z)\): its groups are the compactly supported locally constant chains on the nerve, and its differential is the alternating sum of the face pushforwards.

\begin{lem}\label{lem:boundary}
For every ample groupoid \(\G\) and every \(n\ge 1\) one has \(\partial_{n-1}\circ\partial_n=0\).
In particular, the Moore complex is a chain complex.
\end{lem}

\begin{proof}
For \(n=1\) this holds because \(\partial_0=0\), so let \(n\ge 2\).
Functoriality (\cref{lem:pushforward}\labelcref{it:push-funct}) turns each composite of face pushforwards into the pushforward of a composite, and the face identity \eqref{eq:face-identity} governs those composites.
Concretely,
\begin{equation}\label{eq:boundary-square}
\begin{aligned}
\partial_{n-1}\partial_n
&=\sum_{i=0}^{n-1}\sum_{j=0}^{n}(-1)^{i+j}(d_i)_*(d_j)_*
&&\quad\text{by \eqref{eq:moore-complex}}\\
&=\sum_{i=0}^{n-1}\sum_{j=0}^{n}(-1)^{i+j}(d_i\circ d_j)_*.
&&\quad\text{by \cref{lem:pushforward}\labelcref{it:push-funct}}
\end{aligned}
\end{equation}
Split the index set into \(i<j\) and \(i\ge j\).
On the first block \(d_i\circ d_j=d_{j-1}\circ d_i\) by \eqref{eq:face-identity}, so reindexing by \((a,b)\coloneqq(i,j-1)\) rewrites it as
\begin{equation}\label{eq:boundary-first-block}
-\sum_{0\le a\le b\le n-1}(-1)^{a+b}(d_b\circ d_a)_*.
\end{equation}
The second block is exactly
\begin{equation}\label{eq:boundary-second-block}
+\sum_{0\le a\le b\le n-1}(-1)^{a+b}(d_b\circ d_a)_*.
\end{equation}
The two cancel, so \(\partial_{n-1}\partial_n=0\).
\end{proof}

\begin{defi}\label{def:homology}
The \term{groupoid homology} of an ample groupoid \(\G\) is
\begin{equation}\label{eq:homology}
H_n(\G;\Z)\coloneqq\frac{\ker(\partial_n)}{\im(\partial_{n+1})},
\qquad n\ge 0,
\end{equation}
the homology of its Moore complex.
\end{defi}

For ample groupoids this is the homology introduced by Matui \cite[Definition~3.1]{matui2012homology}.
It reframes, for ample groupoids, the compactly supported homology theory of Crainic and Moerdijk \cite[Section~3]{crainic2000homology}, as explained in \cite[Section~4]{farsi2019ample}.

\section{The Cantor Unit Groupoid and Moore Homology}\label{sec:moore}
Throughout this section \(X\coloneqq\{0,1\}^\N\) carries the product topology.
For \(n\in\N\) and \(\varepsilon\in\{0,1\}^n\) let the cylinder be
\begin{equation}\label{eq:cylinder}
[\varepsilon]\coloneqq\{x=(x_k)_{k\in\N}\in X\mid x_1=\varepsilon_1,\dots,x_n=\varepsilon_n\}.
\end{equation}
Each \([\varepsilon]\) is clopen, and the cylinders form a countable clopen basis of \(X\).

Let \(\G\coloneqq(X\rightrightarrows X)\) be the unit groupoid of \cref{ex:unit}\labelcref{it:ex-unit}, so the only arrows are units and \(r=s=\id_X\).
A composable \(n\)-string of units forces all base points to agree, whence
\begin{equation}\label{eq:unit-nerve}
\G_n=\{(x_1,\dots,x_n)\in X^n\mid x_1=\dots=x_n\}=\{(x,\dots,x)\mid x\in X\}
\end{equation}
for \(n\ge 1\) and \(\G_0=X\).
For \(n\ge 1\) let \(\iota_n\colon X\to\G_n\) and \(\rho_n\colon\G_n\to X\) be the mutually inverse homeomorphisms
\begin{equation}\label{eq:iota-rho}
\iota_n(x)=(x,\dots,x),\qquad
\rho_n(x,\dots,x)=x,\qquad
\rho_n\circ\iota_n=\id_X,
\end{equation}
and put \(\iota_0\coloneqq\rho_0\coloneqq\id_X\) (all are continuous: the components of \(\iota_n\) are \(\id_X\), and \(\rho_n\) is the restriction of a coordinate projection).
The outer faces \(d_0,d_n\) delete an end coordinate.
An inner face \(d_i\) (\(0<i<n\)) multiplies two units over the same point, which again deletes one coordinate.
Either way the output is the constant tuple over the same base point, so for all \(0\le i\le n\), including \(n=1\), where \(d_0=s\) and \(d_1=r\):
\begin{equation}\label{eq:unit-faces}
d_i=\iota_{n-1}\circ\rho_n.
\end{equation}

Since \(X\) is compact and \(\Z\) is discrete, every continuous map \(X\to\Z\) has finite image, hence is locally constant with compact support.
Thus \(C_c(\G_n,\Z)=C(\G_n,\Z)\), and pullback along the homeomorphism \(\rho_n\) furnishes the identifying isomorphism
\begin{equation}\label{eq:rho-iso}
\rho_n^*\colon C(X,\Z)\xrightarrow{\cong}C(\G_n,\Z),
\qquad
\rho_n^*f=f\circ\rho_n,
\qquad
(\rho_n^*)^{-1}=\iota_n^*.
\end{equation}

\begin{lem}\label{lem:faces}
For every \(n\ge 1\) and \(0\le i\le n\), the pushforward \((d_i)_*\colon C_n(\G;\Z)\to C_{n-1}(\G;\Z)\) becomes \(\id_{C(X,\Z)}\) under the identifications \(\rho_n^*,\rho_{n-1}^*\).
\end{lem}

\begin{proof}
It suffices to track a generic \(f\in C(X,\Z)\), identified with \(\rho_n^*f\in C(\G_n,\Z)\), and to show \((d_i)_*(\rho_n^*f)=\rho_{n-1}^*f\).
From \eqref{eq:unit-faces} and functoriality (\cref{lem:pushforward}\labelcref{it:push-funct}), \((d_i)_*=(\iota_{n-1})_*\,(\rho_n)_*\).
Both factors are pushforwards along homeomorphisms, so \cref{lem:pushforward}\labelcref{it:push-homeo} reads them as precomposition with the inverse:
\begin{equation}\label{eq:unit-push-homeo}
(\rho_n)_*=(\rho_n^{-1})^*=\iota_n^*,
\qquad
(\iota_{n-1})_*=(\iota_{n-1}^{-1})^*=\rho_{n-1}^*.
\end{equation}
Applying these in turn,
\begin{equation}\label{eq:unit-push-compute}
\begin{aligned}
(\rho_n)_*(\rho_n^*f)&=(f\circ\rho_n)\circ\rho_n^{-1}=f,\\
(\iota_{n-1})_*f&=f\circ\iota_{n-1}^{-1}=f\circ\rho_{n-1}=\rho_{n-1}^*f.
\end{aligned}
\end{equation}
Hence \((d_i)_*(\rho_n^*f)=\rho_{n-1}^*f\).
The index \(i\) never enters the computation, so all \(n+1\) face pushforwards coincide with \(\id_{C(X,\Z)}\).
\end{proof}

\begin{lem}\label{lem:Moore}
One has \(H_0(\G;\Z)=C(X,\Z)\) and \(H_n(\G;\Z)=0\) for all \(n\ge 1\).
\end{lem}

\begin{proof}
By \cref{lem:faces}, for \(n\ge 1\) every \((d_i)_*\) is \(\id_{C(X,\Z)}\), so
\begin{equation}\label{eq:moore-differential}
\partial_n=\Bigl(\sum_{i=0}^n(-1)^i\Bigr)\id_{C(X,\Z)}
=
\begin{cases}
\id,& n\ \text{even},\\
0,& n\ \text{odd},
\end{cases}
\qquad\text{and}\quad\partial_0=0.
\end{equation}
The complex thus alternates between \(0\) and \(\id\),
\begin{equation}\label{eq:unit-complex}
\cdots\xrightarrow{\id}C(X,\Z)\xrightarrow{0}C(X,\Z)
\xrightarrow{\id}C(X,\Z)\xrightarrow{0}C(X,\Z)\to 0,
\end{equation}
and we read off homology degreewise:
\begin{itemize}
\item \step{Degree \(0\).} Here \(\partial_0=0\) and \(\partial_1=0\), so \(H_0=\ker\partial_0/\im\partial_1=C(X,\Z)/0=C(X,\Z)\).
\item \step{\(n\) even, \(n\ge 2\).} Then \(\partial_n=\id\) is injective, so \(\ker\partial_n=0\) and \(H_n=0\).
\item \step{\(n\) odd, \(n\ge 1\).} Then \(\partial_n=0\), so \(\ker\partial_n=C_n\).
Also \(\partial_{n+1}=\id\) (as \(n+1\) is even), so \(\im\partial_{n+1}=C_n\).
Hence \(H_n=C_n/C_n=0\).
\end{itemize}
Therefore \(H_0(\G;\Z)=C(X,\Z)\) and \(H_n(\G;\Z)=0\) for all \(n\ge 1\).
\end{proof}

\begin{rem}\label{rem:notnew}
\Cref{lem:Moore} is not new.
For a zero-dimensional space, regarded as an ample groupoid consisting of units, the computation appears as \cite[Proposition~4.4]{farsi2019ample}, and it also follows from the identification of the groupoid homology of a space with its compactly supported cohomology \cite[3.5(3)]{crainic2000homology}.
We keep the elementary proof so that both sides of the comparison rest on the same explicit model.
\end{rem}

\begin{rem}\label{rem:homotopy}
The vanishing in positive degrees is witnessed constructively by a contracting homotopy.
Define \(h_n\colon C_n(\G;\Z)\to C_{n+1}(\G;\Z)\) by
\begin{equation}\label{eq:contraction}
h_n\coloneqq
\begin{cases}
\id_{C(X,\Z)},& n\ \text{odd},\\
0,& n\ \text{even}.
\end{cases}
\end{equation}
A case check against \(\partial_n\in\{0,\id\}\) gives, for every \(n\ge 1\),
\begin{equation}\label{eq:contraction-identity}
\partial_{n+1}h_n+h_{n-1}\partial_n=\id_{C_n(\G;\Z)},
\end{equation}
so \(\id\simeq 0\) on the truncation in degrees \(\ge 1\): every positive-degree cycle is a boundary.
The relation fails in degree \(0\) (there \(h_0=0\) and \(h_{-1}=0\)), which is why \(H_0\) survives.
\end{rem}

\section{Classifying Space and Zeroth Singular Homology}\label{sec:classifying}
Throughout this section \(X\) is the Cantor set, \(\G=(X\rightrightarrows X)\) is its unit groupoid, and \(\iota_n,\rho_n\) are the mutually inverse homeomorphisms of \cref{sec:moore}.
The computation that identified the face maps there identifies every structure map of the nerve: for monotone \(\theta\colon[m]\to[n]\), each entry of \(\theta^*(x,\dots,x)\) is a composite of units at \(x\) (\cref{def:nerve}), hence equals \(x\), so
\begin{equation}\label{eq:unit-structure}
\theta^*=\iota_m\circ\rho_n.
\end{equation}
In particular every degeneracy map is \(s_i=\iota_{n+1}\circ\rho_n\), and \(\rho_m\circ\theta^*\circ\iota_n=\id_X\) for every monotone \(\theta\).

\begin{defi}\label{def:const}
For a topological space \(Y\) the \term{constant simplicial space} \(cY\) is the simplicial space with \((cY)_n\coloneqq Y\) for all \(n\ge 0\) and \(\theta^*\coloneqq\id_Y\) for every monotone \(\theta\).
\end{defi}

\begin{lem}\label{lem:iso}
Let \(X\) be the Cantor set, let \(\G=(X\rightrightarrows X)\) be its unit groupoid, and let \(\rho_n\colon\G_n\to X\) be the homeomorphisms of \cref{sec:moore}.
Then \(\rho_\bullet\coloneqq(\rho_n)_{n\ge 0}\colon\G_\bullet\to cX\) is an isomorphism of simplicial spaces.
\end{lem}

\begin{proof}
Each \(\rho_n\) is a homeomorphism.
For monotone \(\theta\colon[m]\to[n]\) the compatibility square
\begin{center}
\begin{tikzcd}[column sep=large,row sep=large]
\G_n\arrow[r,"\theta^*"]\arrow[d,"\rho_n"']&\G_m\arrow[d,"\rho_m"]\\
X\arrow[r,"\id_X"']&X
\end{tikzcd}
\end{center}
commutes:
\begin{equation}\label{eq:iso-square}
\begin{aligned}
\rho_m\circ\theta^*
&=\rho_m\circ\iota_m\circ\rho_n&&\quad\text{by \eqref{eq:unit-structure}}\\
&=\rho_n&&\quad\text{by \eqref{eq:iota-rho}}\\
&=\id_X\circ\rho_n.
\end{aligned}
\end{equation}
Hence \(\rho_\bullet\) is a morphism of simplicial spaces, and since every \(\rho_n\) is a homeomorphism it is an isomorphism.
\end{proof}

\subhead{Geometric Realization}
For \(n\ge 0\) the \term{standard \(n\)-simplex} is
\begin{equation}\label{eq:simplex}
\Delta^n\coloneqq\Bigl\{t=(t_0,\dots,t_n)\in[0,1]^{n+1}\Bigm|\sum_{i=0}^n t_i=1\Bigr\},
\end{equation}
with the Euclidean topology.
For a monotone \(\theta\colon[m]\to[n]\) let \(\theta_*\colon\Delta^m\to\Delta^n\) be the affine map
\begin{equation}\label{eq:affine}
(\theta_*(t))_j\coloneqq\sum_{i\in\theta^{-1}(j)}t_i,
\end{equation}
an empty sum being \(0\).
It does land in \(\Delta^n\), since its entries are nonnegative and the fibers of \(\theta\) partition \([m]\), so \(\sum_{j=0}^n(\theta_*(t))_j=\sum_{i=0}^m t_i=1\).
Being linear in \(t\), it is continuous.
The assignment is functorial: \((\id_{[n]})_*=\id_{\Delta^n}\), and for monotone \(\theta'\colon[l]\to[m]\) and \(\theta\colon[m]\to[n]\), \(t\in\Delta^l\) and \(j\in[n]\),
\begin{equation}\label{eq:affine-functorial}
\begin{aligned}
((\theta\circ\theta')_*(t))_j
&=\sum_{i\in(\theta\theta')^{-1}(j)}t_i&&\quad\text{by \eqref{eq:affine}}\\
&=\sum_{k\in\theta^{-1}(j)}\ \sum_{i\in\theta'^{-1}(k)}t_i&&\quad\text{partition of \((\theta\theta')^{-1}(j)\)}\\
&=(\theta_*(\theta'_*(t)))_j.&&\quad\text{by \eqref{eq:affine}}
\end{aligned}
\end{equation}

Let \(L_\bullet\) be a simplicial space.
The \term{geometric realization} is
\begin{equation}\label{eq:realization}
\real{L_\bullet}\coloneqq\Bigl(\coprod_{n\ge 0}\Delta^n\times L_n\Bigr)\Big/\sim,
\qquad
(\theta_*(t),y)\sim(t,\theta^*(y)),
\end{equation}
where \(\sim\) denotes the equivalence relation generated by these identifications, with \(\theta\) ranging over all monotone maps \([m]\to[n]\) and \((t,y)\) over \(\Delta^m\times L_n\).
Write \(Q\) for the quotient map, \([t,y]\coloneqq Q(t,y)\), and \(\ast_0\) for the unique point of \(\Delta^0\).
The \term{classifying space} of a topological groupoid \(\G\) is \(B\G\coloneqq\real{\G_\bullet}\).

\begin{lem}\label{lem:realfunct}
An isomorphism of simplicial spaces \(f_\bullet\colon L_\bullet\to L'_\bullet\) induces a homeomorphism \(\real{f_\bullet}\colon\real{L_\bullet}\to\real{L'_\bullet}\) with \(\real{f_\bullet}([t,y])=[t,f_n(y)]\) for all \((t,y)\in\Delta^n\times L_n\).
\end{lem}

\begin{proof}
The continuous map \(\coprod_{n\ge 0}(\id_{\Delta^n}\times f_n)\) sends each generating pair to a generating pair,
\begin{equation}\label{eq:realfunct-pairs}
\begin{aligned}
(\theta_*(t),f_n(y))
&\sim(t,\theta^*(f_n(y)))&&\quad\text{by \eqref{eq:realization}}\\
&=(t,f_m(\theta^*(y))),&&\quad\text{by \cref{def:simplicial}}
\end{aligned}
\end{equation}
so by the universal property of the quotient topology it descends to a continuous map \(\real{f_\bullet}\) with the stated formula.
The same construction applied to the inverse morphism \((f_n^{-1})_{n\ge 0}\) gives a continuous map in the other direction, and the two composites fix every class \([t,y]\), so \(\real{f_\bullet}\) is a homeomorphism.
\end{proof}

\begin{lem}\label{lem:const}
For every topological space \(Y\) the maps
\begin{equation}\label{eq:const-maps}
P\colon\real{cY}\to Y,\quad P([t,y])=y,
\qquad
\jmath\colon Y\to\real{cY},\quad\jmath(y)=[\ast_0,y],
\end{equation}
are mutually inverse homeomorphisms.
In particular \(\real{cY}\cong Y\).
\end{lem}

\begin{proof}\keepwithnext
\begin{itemize}
\item \step{\(P\) is well defined and continuous.}
The projections \(\pr_Y\colon\Delta^n\times Y\to Y\) assemble to a continuous map on \(\coprod_n\Delta^n\times Y\) that agrees on each generating pair, since both sides carry the same second coordinate: \(\pr_Y(\theta_*(t),y)=y=\pr_Y(t,y)\).
By the universal property of the quotient it descends to a continuous \(P\colon\real{cY}\to Y\) with \(P([t,y])=y\).
\item \step{\(\jmath\) is continuous.}
It is the composite
\begin{equation}\label{eq:jmath}
Y=\Delta^0\times Y\hookrightarrow\coprod_{n\ge0}\Delta^n\times Y\xrightarrow{Q}\real{cY}
\end{equation}
of the inclusion of the degree-\(0\) summand with the quotient map, hence continuous.
No hypothesis on \(Y\) is used.
In particular, no product is taken with a quotient map.
\item \step{\(P\jmath=\id_Y\).}
For \(y\in Y\) one has \(P(\jmath(y))=P([\ast_0,y])=y\).
\item \step{\(\jmath P=\id_{\real{cY}}\).}
Fix \((t,y)\in\Delta^n\times Y\) and let \(\theta\colon[n]\to[0]\) be the unique monotone map.
Its affine map \(\theta_*\colon\Delta^n\to\Delta^0\) is constant with value \(\ast_0\), and \(\theta^*(y)=y\) in \(cY\), so the generating relation gives
\begin{equation}\label{eq:const-relation}
(\ast_0,y)=(\theta_*(t),y)\sim(t,y),
\qquad\text{that is,}\quad[\ast_0,y]=[t,y].
\end{equation}
Hence \(\jmath(P([t,y]))=\jmath(y)=[\ast_0,y]=[t,y]\), and as the classes \([t,y]\) exhaust \(\real{cY}\) we get \(\jmath P=\id\).
\end{itemize}
Thus \(P\) and \(\jmath\) are mutually inverse continuous bijections, hence homeomorphisms.
\end{proof}

Applying \cref{lem:realfunct} to the isomorphism \(\rho_\bullet\colon\G_\bullet\to cX\) of \cref{lem:iso}, and then \cref{lem:const}, yields homeomorphisms
\begin{equation}\label{eq:BG}
B\G=\real{\G_\bullet}\cong\real{cX}\cong X.
\end{equation}

\subhead{Singular Homology in Degree Zero}

\begin{defi}\label{def:dirsum}
For a set \(I\) the \term{direct sum} is
\begin{equation}\label{eq:dirsum}
\bigoplus_{i\in I}\Z\coloneqq\{a\colon I\to\Z\mid a(i)\neq 0\ \text{for only finitely many }i\in I\},
\end{equation}
an abelian group under pointwise addition.
For \(i\in I\) the \term{standard generator} \(e_i\in\bigoplus_{j\in I}\Z\) is given by \(e_i(i)=1\) and \(e_i(j)=0\) for \(j\neq i\).
Every element \(a\) equals the finite sum \(\sum_{i\in I,\ a(i)\neq 0}a(i)\,e_i\).
\end{defi}

The direct sum is the free abelian group on its standard generators: for every abelian group \(A\) and every map \(\varphi\colon I\to A\) there is exactly one homomorphism \(\bigoplus_{i\in I}\Z\to A\) with \(e_i\mapsto\varphi(i)\), namely \(a\mapsto\sum_{a(i)\neq 0}a(i)\,\varphi(i)\).
It is additive because each \(a\mapsto a(i)\) is, and it is unique because the \(e_i\) generate.

\begin{defi}\label{def:sing}
Let \(Y\) be a topological space.
A \term{singular \(n\)-simplex} in \(Y\) is a continuous map \(\sigma\colon\Delta^n\to Y\).
Write \(S_n(Y)\) for the set of singular \(n\)-simplices.
The group of \term{singular \(n\)-chains} is the free abelian group
\begin{equation}\label{eq:sing-chains}
\Csing_n(Y;\Z)\coloneqq\bigoplus_{\sigma\in S_n(Y)}\Z
\end{equation}
of \cref{def:dirsum}.
We write \(\sigma\) also for the standard generator \(e_\sigma\).
Since \(\Delta^0\) is a point, the singular \(0\)-simplices correspond to the points of \(Y\).
For \(y\in Y\) we write \(\gen{y}\in\Csing_0(Y;\Z)\) for the generator at \(y\).
The \term{singular boundary} \(\dsing_n\colon\Csing_n(Y;\Z)\to\Csing_{n-1}(Y;\Z)\) is the homomorphism determined by
\begin{equation}\label{eq:sing-boundary}
\dsing_n\sigma\coloneqq\sum_{i=0}^n(-1)^i\,\sigma\circ(\delta^n_i)_*
\quad(n\ge 1),
\qquad
\dsing_0\coloneqq 0,
\end{equation}
with \(\Csing_{-1}(Y;\Z)\coloneqq 0\) and \((\delta^n_i)_*\colon\Delta^{n-1}\to\Delta^n\) the affine map of the coface \(\delta^n_i\), and the \term{singular homology} of \(Y\) is
\begin{equation}\label{eq:sing-homology}
\Hsing_n(Y;\Z)\coloneqq\frac{\ker(\dsing_n)}{\im(\dsing_{n+1})},
\qquad n\ge 0.
\end{equation}
\end{defi}

This is the standard singular chain complex \cite[Section~2.1]{hatcher2002algebraic}.
The identity \(\dsing_{n-1}\circ\dsing_n=0\) follows exactly as in \cref{lem:boundary}, with functoriality of \(\theta\mapsto\theta_*\) and \cref{lem:coface} in place of \cref{lem:pushforward}\labelcref{it:push-funct} and \cref{lem:nerve}.
Write \(v_0\coloneqq(1,0)\) and \(v_1\coloneqq(0,1)\) for the two vertices of \(\Delta^1\).
Then \((\delta^1_0)_*(\ast_0)=v_1\) and \((\delta^1_1)_*(\ast_0)=v_0\), so the boundary of a singular \(1\)-simplex \(\sigma\) is
\begin{equation}\label{eq:boundary-one}
\dsing_1\sigma=\gen{\sigma(v_1)}-\gen{\sigma(v_0)}.
\end{equation}
A continuous map \(u\colon Y\to Y'\) induces homomorphisms \(\Csing_n(Y;\Z)\to\Csing_n(Y';\Z)\) given by \(\sigma\mapsto u\circ\sigma\), which commute with the boundaries, because \((u\circ\sigma)\circ(\delta^n_i)_*=u\circ\bigl(\sigma\circ(\delta^n_i)_*\bigr)\).
The assignment is functorial in \(u\) and is the identity for \(u=\id_Y\), so a homeomorphism induces isomorphisms \(\Hsing_n(Y;\Z)\cong\Hsing_n(Y';\Z)\) in every degree.
Combined with \eqref{eq:BG} this gives \(\Hsing_\bullet(B\G;\Z)\cong\Hsing_\bullet(X;\Z)\) for the Cantor unit groupoid.

\begin{lem}\label{lem:H0}
For every topological space \(Y\) the group \(\Hsing_0(Y;\Z)\) is isomorphic to \(\bigoplus_{C\in\pi_0(Y)}\Z\), where \(\pi_0(Y)\) denotes the set of path components of \(Y\).
\end{lem}

\begin{proof}
Two points of \(Y\) lie in the same \term{path component} if some continuous map \([0,1]\to Y\) sends \(0\) to the one and \(1\) to the other.
This is an equivalence relation: constant paths give reflexivity, precomposition with \(t\mapsto 1-t\) gives symmetry, and concatenation of paths, continuous because \([0,1/2]\) and \([1/2,1]\) form a finite closed cover of \([0,1]\), gives transitivity.
The set of classes is \(\pi_0(Y)\), and \(\bar y\in\pi_0(Y)\) denotes the class of \(y\in Y\).
Since \(\dsing_0=0\), the group \(\Hsing_0(Y;\Z)\) is \(\Csing_0(Y;\Z)/\im\dsing_1\).
Write \(\equiv\) for congruence modulo \(\im\dsing_1\).
\begin{itemize}
\item \step{Endpoints of \(1\)-simplices.}
If \(y,y'\in Y\) lie in the same path component, there is a continuous \(\sigma\colon\Delta^1\to Y\) with \(\sigma(v_0)=y\) and \(\sigma(v_1)=y'\): precompose a path with the homeomorphism \(\Delta^1\to[0,1]\) given by \((t_0,t_1)\mapsto t_1\), which sends \(v_0\) to \(0\) and \(v_1\) to \(1\).
Then \(\dsing_1\sigma=\gen{y'}-\gen{y}\) by \eqref{eq:boundary-one}, so \(\gen{y}\equiv\gen{y'}\).
Conversely, for every singular \(1\)-simplex \(\sigma\) the points \(\sigma(v_0)\) and \(\sigma(v_1)\) lie in the same path component, since \(t\mapsto\sigma(1-t,t)\) is a path between them.
\item \step{The map \(\bar\beta\).}
Let \(\beta\colon\Csing_0(Y;\Z)\to\bigoplus_{C\in\pi_0(Y)}\Z\) be the unique homomorphism with \(\beta(\gen{y})=e_{\bar y}\), which exists by the universal property stated after \cref{def:dirsum}.
For every singular \(1\)-simplex \(\sigma\),
\begin{equation}\label{eq:beta-boundary}
\begin{aligned}
\beta(\dsing_1\sigma)
&=\beta(\gen{\sigma(v_1)})-\beta(\gen{\sigma(v_0)})&&\quad\text{by \eqref{eq:boundary-one}}\\
&=e_{\overline{\sigma(v_1)}}-e_{\overline{\sigma(v_0)}}&&\quad\text{definition of \(\beta\)}\\
&=0,
\end{aligned}
\end{equation}
because the two endpoints lie in the same path component.
Hence \(\beta\) vanishes on \(\im\dsing_1\) and induces a homomorphism \(\bar\beta\colon\Hsing_0(Y;\Z)\to\bigoplus_{C\in\pi_0(Y)}\Z\).
\item \step{The map \(\alpha\).}
Choose a point \(y_C\in C\) for every \(C\in\pi_0(Y)\), and let \(\alpha\colon\bigoplus_{C\in\pi_0(Y)}\Z\to\Hsing_0(Y;\Z)\) be the unique homomorphism with \(\alpha(e_C)=\gen{y_C}+\im\dsing_1\), again by the universal property.
\item \step{Mutually inverse.}
On generators, \(\bar\beta(\alpha(e_C))=\beta(\gen{y_C})=e_{\overline{y_C}}=e_C\), so \(\bar\beta\circ\alpha=\id\).
For \(y\in Y\) the points \(y\) and \(y_{\bar y}\) lie in the same path component, so \(\gen{y}\equiv\gen{y_{\bar y}}\) and
\begin{equation}\label{eq:alpha-beta}
\begin{aligned}
\alpha\bigl(\bar\beta(\gen{y}+\im\dsing_1)\bigr)
&=\alpha(e_{\bar y})&&\quad\text{definition of \(\bar\beta\)}\\
&=\gen{y_{\bar y}}+\im\dsing_1&&\quad\text{definition of \(\alpha\)}\\
&=\gen{y}+\im\dsing_1.&&\quad\text{\(\gen{y}\equiv\gen{y_{\bar y}}\)}
\end{aligned}
\end{equation}
The classes of the generators \(\gen{y}\) generate \(\Hsing_0(Y;\Z)\), so \(\alpha\circ\bar\beta=\id\).
Hence \(\alpha\) and \(\bar\beta\) are mutually inverse isomorphisms.\qedhere
\end{itemize}
\end{proof}

Since \(X\) is totally disconnected, the image of every path \([0,1]\to X\) is connected, hence a singleton.
Therefore every path component of \(X\) is a singleton.
We identify \(\pi_0(X)=X\), and \cref{lem:H0} together with \eqref{eq:BG} yields
\begin{equation}\label{eq:H0-cantor}
\Hsing_0(B\G;\Z)\cong\Hsing_0(X;\Z)\cong\bigoplus_{x\in X}\Z.
\end{equation}

\begin{rem}\label{rem:onlydeg0}
Total disconnectedness forces every singular simplex \(\Delta^n\to X\) to be constant, since \(\Delta^n\) is connected, so \(\Csing_n(X;\Z)\) is free on the points of \(X\) and each face sends the constant \(n\)-simplex at \(x\) to the constant \((n-1)\)-simplex at \(x\).
The singular boundary is thus \(\dsing_n=\bigl(\sum_{i=0}^n(-1)^i\bigr)\id\) for \(n\ge 1\), together with \(\dsing_0=0\): exactly the alternating pattern of the Moore complex in \cref{lem:Moore}.
Consequently
\begin{equation}\label{eq:sing-vanish}
\Hsing_n(B\G;\Z)\cong\Hsing_n(X;\Z)=0=H_n(\G;\Z)\qquad\text{for } n\ge 1,
\end{equation}
so the two theories agree in every positive degree.
They part company only in degree \(0\), where the degree-\(0\) groups are \(C(X,\Z)\) for the Moore complex and \(\bigoplus_{x\in X}\Z\) for the singular complex: the same complex shape, but a countable group against one of cardinality \(2^{\aleph_0}\).
\end{rem}

\section{Cardinality Separation}\label{sec:cardinality}
Recall from \cref{def:dirsum} the direct sum \(\bigoplus_{x\in X}\Z\) and its standard generators \(e_x\).
The assignment \(x\mapsto e_x\) is injective, and \(\abs{X}=\abs{\{0,1\}^{\N}}=2^{\aleph_0}\), so
\begin{equation}\label{eq:card-lower}
\abs[\Big]{\bigoplus_{x\in X}\Z}\ge\abs{X}=2^{\aleph_0}.
\end{equation}

We bound \(C(X,\Z)\) from the other side.
The argument rests on the elementary combinatorics of cylinders: for \(\varepsilon\in\{0,1\}^m\) and \(\varepsilon'\in\{0,1\}^n\) with \(m\le n\),
\begin{equation}\label{eq:cylinder-dichotomy}
[\varepsilon]\cap[\varepsilon']\neq\varnothing
\iff
\varepsilon\ \text{is a prefix of}\ \varepsilon',
\quad\text{in which case}\quad
[\varepsilon']\subseteq[\varepsilon].
\end{equation}
Otherwise \([\varepsilon]\) and \([\varepsilon']\) are disjoint.
In particular the \(2^N\) \term{level-\(N\) cylinders} \(\{[\varepsilon']\mid\varepsilon'\in\{0,1\}^N\}\) form a clopen partition of \(X\), and each refines the partition at every lower level.

\begin{lem}\label{lem:resolution}
For every \(f\in C(X,\Z)\) there exists \(N\in\N\) such that \(f\) is constant on each level-\(N\) cylinder.
\end{lem}

\begin{proof}\keepwithnext
\begin{itemize}
\item \step{Local trivialization.}
Fix \(f\in C(X,\Z)\).
Since \(f\) is locally constant and the cylinders form a basis, each \(x\in X\) admits a cylinder \([\varepsilon_x]\ni x\) on which \(f\equiv f(x)\).
\item \step{Compactness.}
The family \(\{[\varepsilon_x]\}_{x\in X}\) is an open cover of the compact space \(X\), so it has a finite subcover \([\varepsilon_{x_1}],\dots,[\varepsilon_{x_r}]\).
Put \(N\coloneqq\max_{1\le j\le r}\abs{\varepsilon_{x_j}}\).
\item \step{Constancy at level \(N\).}
Fix \(\varepsilon'\in\{0,1\}^N\).
The cylinder \([\varepsilon']\) is nonempty, and since the \([\varepsilon_{x_j}]\) cover \(X\) it meets some \([\varepsilon_{x_j}]\).
As \(\abs{\varepsilon_{x_j}}\le N\), the dichotomy \eqref{eq:cylinder-dichotomy} forces \([\varepsilon']\subseteq[\varepsilon_{x_j}]\), and there \(f\equiv f(x_j)\).
Hence \(f\) is constant on \([\varepsilon']\), with \(\varepsilon'\) arbitrary.\qedhere
\end{itemize}
\end{proof}

\begin{lem}\label{lem:countable}
The group \(C(X,\Z)\) is countably infinite.
\end{lem}

\begin{proof}
For \(N\ge 0\) set
\begin{equation}\label{eq:RN}
R_N\coloneqq\{f\in C(X,\Z)\mid f\ \text{is constant on each level-\(N\) cylinder}\}.
\end{equation}
\begin{itemize}
\item \step{Each \(R_N\) is countable.}
Evaluation on the (finitely many) level-\(N\) cylinders gives a homomorphism \(E_N\colon R_N\to\Z^{\{0,1\}^N}\), where \(E_N(f)(\varepsilon')\) is the value of \(f\) on the cylinder \([\varepsilon']\), well defined as \(f\) is constant there.
It is an isomorphism: its inverse sends \(g\colon\{0,1\}^N\to\Z\) to
\begin{equation}\label{eq:RN-inverse}
\sum_{\varepsilon'\in\{0,1\}^N}g(\varepsilon')\,1_{[\varepsilon']},
\end{equation}
where \(1_{[\varepsilon']}\) has value \(1\) on \([\varepsilon']\) and \(0\) elsewhere.
This sum is locally constant, hence lies in \(R_N\), and it is well defined because the level-\(N\) cylinders partition \(X\).
The two composites are the identity by evaluation on cylinders.
Since \(\{0,1\}^N\) is finite, \(\Z^{\{0,1\}^N}\cong\Z^{2^N}\) is countable, hence so is \(R_N\).
\item \step{Exhaustion.}
Refinement gives \(R_N\subseteq R_{N+1}\), and \cref{lem:resolution} gives \(C(X,\Z)=\bigcup_{N\ge0}R_N\).
Thus \(C(X,\Z)\) is a countable union of countable groups, hence countable.
It contains the constants \(\Z\), so it is countably infinite, that is, \(\abs{C(X,\Z)}=\aleph_0\).\qedhere
\end{itemize}
\end{proof}

By \cref{lem:countable}, \(\abs{C(X,\Z)}=\aleph_0\), whereas \(\abs[\big]{\bigoplus_{x\in X}\Z}\ge 2^{\aleph_0}\) by \eqref{eq:card-lower}.
The two groups have different cardinalities, so there is no bijection between them, a fortiori no group isomorphism:
\begin{equation}\label{eq:noniso}
C(X,\Z)\not\cong\bigoplus_{x\in X}\Z.
\end{equation}
Together with \(H_0(\G;\Z)=C(X,\Z)\) and \(H_n(\G;\Z)=0\) for \(n\ge 1\) (\cref{lem:Moore}) and \(\Hsing_0(B\G;\Z)\cong\bigoplus_{x\in X}\Z\) \eqref{eq:H0-cantor}, this proves \cref{prop:main}.

The example locates the mismatch.
In degree \(0\) the Matui-type theory sees locally constant integer-valued functions on the unit space, the singular theory path components.
On a totally disconnected space the first group is countable and the second is free of rank \(2^{\aleph_0}\), so the two invariants come apart at the first degree where either of them is nonzero.
Above that degree they agree, by \cref{rem:onlydeg0}.

\begin{acknowledgements}
The author is grateful to the anonymous reviewer, and thanks Henning Urbat and Stefan Milius for their advice.
\end{acknowledgements}

\end{document}